\documentclass[11pt,a4paper]{article}

\usepackage[british]{babel}
\usepackage[utf8]{inputenc}

\usepackage{tabularx}
\usepackage{caption}

\usepackage{amssymb,amsmath,amsthm,thmtools}
	\numberwithin{equation}{section}
\usepackage{mathtools}	
\usepackage{accents}
\usepackage{multirow}
\usepackage{graphicx}
\usepackage{amscd}
\usepackage{epic, eepic}
\usepackage{url}
\usepackage{xcolor}
\usepackage{todonotes}
\usepackage{enumerate}
\usepackage{enumitem}
\usepackage{comment}
\usepackage{hhline}
\usepackage{dsfont} 
\usepackage[hidelinks]{hyperref}
    \addto\extrasbritish{%
    }

\usepackage{tikz-cd}

\newcommand{\PG}{\mathrm{PG}}
\newcommand{\AG}{\mathrm{AG}}

\newcommand{\ff}{\mathbb{F}}

\newcommand{\zz}{\mathbb{Z}}

 \makeatletter
 \@addtoreset{equation}{section}
 \makeatother

 \makeatletter
 \@addtoreset{table}{section}
 \makeatother
 
\declaretheorem[style=plain,name=Theorem,numberwithin=section]{theorem}

\declaretheorem[style=plain,name=Lemma,sibling=theorem]{lemma}

\declaretheorem[style=plain,name=Corollary,sibling=theorem]{corollary}

\declaretheorem[style=plain,name=Observation,sibling=theorem]{observation}

\declaretheorem[style=definition,name=Definition,sibling=theorem]{definition}

\declaretheorem[style=definition,name=Remark,sibling=theorem]{remark}

    \makeatletter
    \providecommand*{\cupdot}{%
      \mathbin{%
        \mathpalette\@cupdot{}%
      }%
    }
    \newcommand*{\@cupdot}[2]{%
      \ooalign{%
        $\m@th#1\cup$\cr
        \sbox0{$#1\cup$}%
        \dimen@=\ht0 %
        \sbox0{$\m@th#1\cdot$}%
        \advance\dimen@ by -\ht0 %
        \dimen@=.5\dimen@
        \hidewidth\raise\dimen@\box0\hidewidth
      }%
    }
    
    \providecommand*{\bigcupdot}{%
      \mathop{%
        \vphantom{\bigcup}%
        \mathpalette\@bigcupdot{}%
      }%
    }
    \newcommand*{\@bigcupdot}[2]{%
      \ooalign{%
        $\m@th#1\bigcup$\cr
        \sbox0{$#1\bigcup$}%
        \dimen@=\ht0 %
        \advance\dimen@ by -\dp0 %
        \sbox0{\scalebox{2}{$\m@th#1\cdot$}}%
        \advance\dimen@ by -\ht0 %
        \dimen@=.5\dimen@
        \hidewidth\raise\dimen@\box0\hidewidth
      }%
    }
    \makeatother

\usepackage{pgfplots}
\pgfplotsset{compat=1.15}
\usepackage{mathrsfs}
\usetikzlibrary{arrows}

\definecolor{qqqqff}{rgb}{0.,0.,1.}
\definecolor{ffqqqq}{rgb}{1.,0.,0.}
\definecolor{yqyqyq}{rgb}{0.5019607843137255,0.5019607843137255,0.5019607843137255}

\makeatletter
\newcommand{\maketitlealph}{
  \begingroup
  \renewcommand{\@fnsymbol}[1]{\@alph{##1}}
  \maketitle
  \endgroup
}
\makeatother
\title{A superlinear improvement on line-free sets in $\ff_p^3$}
\author{Benedek Kovács\thanks{Alfréd Rényi Institute of Mathematics, Reáltanoda utca 13-15, Budapest 1053, Hungary. E-mail: {\tt kovacs.benedek@renyi.hu}}
}
\date{}

\begin{document}
	
    \maketitlealph
	
	\begin{abstract}
    Building on an earlier result of the author together with Elsholtz, Führer, Füredi, Pach, Simon and Velich, we present an improved construction for a line-free set in $\ff_p^3$, having size at least $(p-1)^3+\frac18 p^{3/2} - O(p)$ as $p\to \infty$. This results in a superlinear-term improvement over the standard hypercube construction $\{0,1,\ldots,p-2\}^3$. Via complementation, a line-free set in $\ff_p^3$ corresponds to a $2$-blocking set in the affine geometry $\AG(3,p)$, hence we also obtain an upper bound of $3p^2-\frac18p^{3/2}+O(p)$ on the smallest size of such a $2$-blocking set.
	\end{abstract}

    \medskip
\noindent\textbf{Keywords:} blocking sets, line-free sets, affine geometry

\section{Introduction}

In this manuscript, we investigate the following extremal combinatorial problem: what is the maximum possible size of a subset $S$ of $\ff_p^3$ which does not contain a full line? Here $\ff_p^3$ is the $3$-dimensional vector space over the $p$-element field for some prime $p$, corresponding to the points of the affine geometry $\mathrm{AG}(3,p)$.

More generally, we denote by $r_k(\ff_p^n)$ the largest size of a subset $S\subseteq \ff_p^n$ that does not contain any nontrivial $k$-term arithmetic progression (that is, one with nonzero difference); note that for $k=p$ this condition is equivalent to saying that $S$ does not contain any full line.

The famous \textit{cap set problem} investigates the value of $r_3(\ff_3^n)$; note that three distinct points $\mathbf{a},\mathbf{b},\mathbf{c}\in \ff_3^n$ form a line if and only if $\mathbf{a}+\mathbf{b}+\mathbf{c}=\mathbf{0}$. Ellenberg and Gijswijt \cite{Ellenberg2017} proved that $r_3(\ff_3^n)=o(2.756^n)$ using a polynomial technique introduced by Croot, Lev and Pach \cite{Croot2017}, and Jiang \cite{Jiang2021} later improved this upper bound by a factor of $\sqrt{n}$, yielding the best known improvement so far. 

Regarding lower bounds, Edel \cite{Edel2004} developed a product-type construction of cap sets and obtained a bound of $\Omega(2.2173^n)$, using so-called \textit{admissible sets} to construct large higher-dimensional cap sets from lower-dimensional ones. The bound was further improved to $\Omega(2.218^n)$ by Tyrrell \cite{Tyrrell2023} using an extension of the construction and a SAT solver. Recently, large language models were utilized to search for better admissible sets, and hence Romera-Paredes et al. \cite{RomeraParedes2024} obtained a bound of $\Omega(2.2202^n)$ using FunSearch, and Zhai et al. \cite{Zhai2025} obtained $\Omega(2.2203^n)$ using $X$-evolve.

For further recent work on the value of $r_k(\ff_p^n)$, see \cite{Elsholtz2024,Fuhrer2026}; note that these papers consider the case where $n$ is sufficiently large in terms of $p$.

In the general case when $k=p$, the hypercube construction $\{0,1,\ldots,p-2\}^n$ gives a lower bound of $r_p(\ff_p^n)\ge (p-1)^n$. For $n=1$ this is obviously sharp, and for $n=2$ its sharpness was proven by Jamison \cite{Jamison1977} and Brouwer and Schrijver \cite{Brouwer1978} using the polynomial method. However, for $n=3$, better constructions exist. The author together with Elsholtz, Führer, Füredi, Pach, Simon and Velich \cite{Elsholtz2023} improved the lower bound by applying a slight modification to the hypercube construction involving $\Theta(p)$ points. Regarding the upper bound, we also showed an improvement over a previous bound of $r_p(\ff_p^3)\le p^3-2p^2+1$ by Sziklai \cite{Sziklai1997}, using a double counting argument for point pairs on planes.

\begin{theorem}[Elsholtz, Führer, Füredi, Kovács, Pach, Simon, Velich \cite{Elsholtz2023}]\label{thm:earlier}
For all primes $p\ge 5$, we have
 $$p^3-2p^2-(\sqrt{2}-1)p+2 \ge r_p(\ff_p^3)\ge p^3-3p^2+4p-2\sqrt{p}-1.$$
\end{theorem}

As a comparison to the hypercube bound $(p-1)^3$, note that here the lower bound is equal to $(p-1)^3+p-2\sqrt{p}$. To the author's knowledge, there have been no further improvements from either side, so the precise coefficient of the $p^2$ term in $r_p(\ff_p^3)$ remains an open question.

In this paper, we refine the earlier line-free construction in \cite{Elsholtz2023}, this time modifying the hypercube $[0,p-2]^3$ in $\Theta(p^{3/2})$ points. Thus we exhibit a superlinear improvement over the lower bound on $r_p(\ff_p^3)$ in Theorem \ref{thm:earlier}:

\begin{theorem}\label{thm:mainthm}
For all primes $p\ge 3$, we have $$r_p(\ff_p^3)\ge p^3-3p^2+\frac18p^{3/2}-\frac{7}{16}p.$$ As a consequence, we have $$r_p(\ff_p^3)\ge (p-1)^3+\frac18p^{3/2}-O(p)$$ as the prime $p$ tends to infinity.
\end{theorem}

In finite geometry, a \textit{$d$-blocking set} $B$ of an affine or projective space $T$ is a set of points $B\subseteq T$ that intersects every subspace of $T$ of codimension $d$ in at least one point. Note that $B$ is a $d$-blocking set if and only if its complement $T\setminus B$ does not contain any full subspace of codimension $d$. For $d=n-1$, this means that $T\setminus B$ is line-free. So our main result, Theorem \ref{thm:mainthm}, can be stated in terms of $2$-blocking sets in the affine space $\ff_p^3$:

\begin{definition}
For integers $n,d,q$ such that $n\ge 2$, $~1\le d\le n-1$ and $q$ is a prime power, let $a_d(n,q)$ denote the smallest size of a $d$-blocking set in the affine geometry $\AG(n,q)$.
\end{definition}
\begin{corollary}
For primes $p$, $$a_2(3,p)\le 3p^2-\frac18p^{3/2}+O(p)$$ holds as $p\to \infty$.
\end{corollary}

\section{Constructing our line-free set}

The goal of this section is to prove Theorem \ref{thm:mainthm}. Throughout the section, $p\ge 3$ is a fixed prime, and we identify $\ff_p$ with the set $\{0, 1, \ldots, p-1\}$ as usual, inheriting the natural ordering from $\zz$. In $\ff_p^n$, a \textit{line} always refers to an affine line, that is, a subset of the form $\{\mathbf{a}+\lambda\mathbf{v}: \lambda\in \ff_p\}$ for some $\mathbf{a},\mathbf{v}\in \ff_p^n$ with $\mathbf{v}\ne\mathbf{0}$. For $m\le n$, the notation $[m,n]$ denotes the set $\{m, m+1, \ldots, n\}$, where $m$ and $n$ are integers or elements of $\ff_p$.

To prove Theorem \ref{thm:mainthm}, we improve upon the construction in \cite{Elsholtz2023} by the author together with Elsholtz, Führer, Füredi, Pach, Simon and Velich (see Theorem \ref{thm:earlier}). Similarly to that construction, the layers of our $3$-dimensional line-free set will be constructed using complements of planar blocking sets. Namely, we will make use of the following observation:

\begin{observation}\label{obs:planarblocking}
Let $L$ be any line in $\ff_p^2$. If $B\subseteq \ff_p^2$ is a set that contains $L$ and meets every line parallel to $L$ in at least one point, then $B$ is a $1$-blocking set in $\ff_p^2$.
\end{observation}
\begin{proof}
For any line $L'\subseteq \ff_p^2$, if $L'\cap L\ne \emptyset$ then $L'$ meets $B$ as well. Otherwise, $L'$ is parallel to $L$, hence contains at least one point of $B$ by our assumption.
\end{proof}

Observe that by the result of Jamison \cite{Jamison1977}, and Brouwer and Schrijver \cite{Brouwer1978}, $r_p(\ff_p^2)=(p-1)^2$, or equivalently, $a_1(2,p)=2p-1$. Therefore, a set $B$ in Observation \ref{obs:planarblocking} that meets every line parallel to $L$ in exactly one point, which hence has $2p-1$ points in total, is actually a minimum-size $1$-blocking set in $\ff_p^2$. One specific example is the set $\ff_p^2\setminus [0,p-2]^2$, where we call $[0,p-2]^2$ the \textit{standard square}, a special case of the hypercube construction.

Our new construction will be made up from $p$ layers, each of which is a line-free set in $\ff_p^2$. Most layers will be standard squares, except for $\ell+1$ of them, where $\ell=\left\lfloor \frac14\sqrt{p}\right\rfloor$. In $\ell$ further layers, which we call the \textit{special layers}, we will exclude the points of a $1$-blocking set as in Observation \ref{obs:planarblocking}. Here $L$ is the main diagonal, and the further points of the $1$-blocking set form a grid of $r$ columns and $(s+1)$ rows, where $r=\left\lfloor \sqrt{p}\right\rfloor$ and $s=\left\lfloor \frac{p-2}{r}\right\rfloor$. Throughout the special layers, the grids will use the same rows, but the sets of columns will be pairwise disjoint. We will also need to exclude some of the points from the column $y=p-1$ and the row $z=p-1$ in the special layers. In the final layer, we will only include those points in our line-free set that were included in the grid in some special layer.

The following lemma defines these grids in the special layers and proves that they indeed give $1$-blocking sets:

\begin{lemma}\label{lem:blocking}
Let $r=\left \lfloor \sqrt{p}\right \rfloor$ and $s=\left \lfloor \frac{p-2}{r}\right \rfloor$. Let $0\le t\le p-r$ be any integer. Then in $\ff_p^2$, the union of the main diagonal $\{(x,x): x\in \ff_p\}$ and the grid $[t, t+r-1]\times \{0, r, 2r, \ldots, sr, p-1\}$ is a $1$-blocking set.
\end{lemma}
\begin{proof}
Let $B$ be the described set. We will apply Observation \ref{obs:planarblocking}, where $L$ is taken to be the main diagonal. It suffices to show that $B$ contains at least one point from each line of the form $x-y=c$, where $c\in \ff_p\setminus \{0\}$. If $x$ ranges over $[t, t+r-1]$ and $y$ ranges over $\{0, r, 2r, \ldots, sr\}$, then the set of integer values $x-y$ takes is $[t-sr, t+r-1]$, which is an interval containing $(s+1)r$ consecutive integers. Here $(s+1)r=\left(\left\lfloor\frac{p-2}{r}\right\rfloor+1\right)r>\frac{p-2}{r}\cdot r=p-2$, meaning that $(s+1)r\ge p-1$. So over $\ff_p$, we either get every value, or there will be one missing value, $x-y=t+r$, which is given by the point $(t+r-1, p-1)\in B$.
\end{proof}

Note that this type of blocking set already appeared in the construction in \cite{Elsholtz2023}, where in the complement of the line-free set, layer $p-2$ was essentially an instance of the set in Lemma \ref{lem:blocking}. In the special layers of our new construction, we will vary the value of $t$ so that the intervals $[t, t+r-1]$ of columns are pairwise disjoint.

\begin{proof}[Proof of Theorem \ref{thm:mainthm}]
We may assume that $p\ge 17$, as otherwise one may check that $(p-1)^3\ge p^3-3p^2+\frac18p^{3/2}-\frac{7}{16}p$ holds, and the hypercube construction already proves the claimed bound.

Fix $p$, and set $r=\left \lfloor \sqrt{p}\right \rfloor$, $s=\left \lfloor \frac{p-2}{r}\right \rfloor$, and $\ell=\left\lfloor \frac14\sqrt{p}\right\rfloor$. Note that the integers $r,s,\ell$ satisfy $r\ge 4$ and $s,\ell\ge 1$.

For a set $T\subseteq \ff_p^3$, and a value $i\in \ff_p$, we define the \textit{$i$-th layer} of $T$ as
$$T_i:=\{(y,z)\in \ff_p^2: (i,y,z)\in T\}.$$

We first provide an initial construction $S^*\subseteq \ff_p^3$ which might contain full lines, which we will later slightly modify to obtain our line-free set $S$. The layers of the set $S^*$ are defined as follows:

\begin{itemize}
\item For each $i\in [0, p-\ell-2]$, let the $i$-th layer of $S^*$ be a standard square, that is, $S^*_i:=[0,p-2]^2$.

\item For each $i\in [p-\ell-1, p-2]$, i.e. for special layers, let $S^*_i:=\ff_p^2\setminus A_i$, where
\[
\begin{aligned}
A_i:={}&\Bigl\{(y,y):y\in\mathbb F_p\Bigr\}\\
&{}\cup
\Bigl(\bigl([t,t+r-1]\cup\{p-1\}\bigr)
      \times\{0,r,2r,\ldots,sr\}\Bigr)\\
&{}\cup
\Bigl([0,\ell r-1]\times\{p-1\}\Bigr),
\end{aligned}
\]

where we set $t=(i-(p-\ell-1))r$.

Observe that $A_i$ contains an instance of the set from Lemma \ref{lem:blocking}. Indeed, $i-(p-\ell-1)\in [0,\ell-1]$, so $t\in [0, (\ell-1)r]$ and $[t, t+r-1] \times \{p-1\} \subseteq [0, \ell r-1]\times \{p-1\}$. Note that $t\le (\ell-1)r\le p-r$ holds since $\ell r\le \frac14\sqrt{p}\cdot \sqrt{p}\le p$.

Therefore, $A_i$ is a $1$-blocking set in $(\ff_p^3)_i=\{(i,y,z): y,z\in \ff_p\}$, meaning that there are no lines within $S^*_i$. 

\item Finally, let $S^*_{p-1}:=[0,\ell r-1] \times \{0, r, 2r, \ldots, sr\}$.
\end{itemize}

Now we prove that all lines contained within $S^*$ belong to a specific family, and we demonstrate a set of points whose removal from $S^*$ yields a line-free set.

\begin{itemize}

\item Observe that there are no full lines within a layer. For $i\in [0, p-\ell-2]$, $S^*_i$ is the standard square, which is a line-free set. We already saw that $S^*_i$ contains no lines for $i\in [p-\ell-1, p-2]$ either; finally, $S^*_{p-1}$ is a subset of the standard square because $\ell r-1=\left\lfloor \frac14\sqrt{p}\right\rfloor \cdot \left\lfloor \sqrt{p}\right\rfloor -1\le \frac14p-1<p-1$ and $sr=\left\lfloor \frac{p-2}{r}\right\rfloor \cdot r\le p-2<p-1$.

\item There are no full lines of the form $L_{j,k}=\{(x,j,k): x\in \ff_p\}$ in $S^*$. If there were such a line, then we would have $(j,k)\in S^*_{p-1}$ for some $j=ar+b$ and $k=cr$, where $a\in [0,\ell-1]$, $b\in [0,r-1]$ and $c\in [0,s]$. However, in this case, the point $(j,k)$ also appears in the set $A_i$ for $i=(p-\ell-1)+a \in [p-\ell-1, p-2]$, meaning that $(i,j,k)\not \in S^*$.

\item There are also no full lines in $S^*$ where the $y$-coordinate is constant ($y=j$) but the $x$- and $z$-coordinates vary. Suppose to the contrary that there is such a line $L$ with $(p-1,j,k)\in L$. Then $j\in [0, \ell r-1]$ and $k\le sr<p-1$. Since the $z$-coordinate must take every value once, there exists $i\in [0,p-2]$ such that $(i,j,p-1)\in L$. However there cannot actually be any such $i$, since $S^*_i$ is a standard square for all $0\le i\le p-\ell-2$, and hence does not contain any point with $z$-coordinate $p-1$, and for $i\in [p-\ell-1, p-2]$ we have included all points $(j,p-1)$ with $j\in [0, \ell r-1]$ in $A_i$ and hence excluded them from $S^*_i$.

\item Similarly, there are no full lines in $S^*$ whose $z$-coordinate is constant ($z=k$) but the $x$- and $y$-coordinates vary, because layers $S^*_i$ for $i\in [0,p-\ell-2]\cup \{p-1\}$ do not contain any points $(y,z)$ with $y=p-1$, and for all $i\in [p-\ell-1, p-2]$ we have included all points $(p-1,k)$ with $k\in \{0, r, 2r, \ldots, sr\}$ in $A_i$ and hence excluded them from $S^*_i$.

\item The last case is that we have a full line $L\subseteq S^*$ in which all three coordinates vary. Let the unique point of $L$ in layer $p-1$ be $P_0=(p-1,j,k)$, where $j\in [0, \ell r-1]$ and $k\in\{0, r, 2r, \ldots, sr\}$. Since $j,k\ne p-1$, the line $L$ must contain a point with $y$-coordinate $p-1$ in some layer $i_1\ne p-1$, and a point with $z$-coordinate $p-1$ in some layer $i_2\ne p-1$. Note that these two layers cannot coincide, as the point $(p-1,p-1)$ is excluded from every layer. Also, as layers $0$ to $p-\ell-2$ are standard squares, we have $i_1,i_2\in [p-\ell-1, p-2]$.

Let $P_1=(i_1, p-1, k_1)$ and $P_2=(i_2, j_2, p-1)$ be these two points on $L$. By collinearity of $P_0$, $P_1$ and $P_2$, the $x$- and $z$-coordinates of these three points satisfy the following equation in $\ff_p$:
$$\frac{i_1-i_2}{(p-1)-i_2}=\frac{k_1-(p-1)}{k-(p-1)}.$$

This means that we can express $k_1$ in $\ff_p$ as follows:
$$k_1=(p-1)+\frac{i_1-i_2}{(p-1)-i_2}\cdot (k-(p-1)).$$

For every value of $i_1\in [p-\ell-1, p-2]$, the value of $i_2$ can take at most $\ell-1$ possible values ($i_2\in [p-\ell-1, p-2]\setminus \{i_1\}$), and $k$ can take at most $s+1$ possible values ($k\in\{0, r, 2r, \ldots, sr\}$). Hence for each fixed $i_1$, the number of possible values for $k_1$ is at most $(\ell-1)(s+1)$. If we remove all corresponding points $P_1=(i_1, p-1, k_1)$ from $S^*$, we hence get a line-free set. As there are $\ell$ possible values of $i_1$, this requires deleting at most $\ell(\ell-1)(s+1)$ points.
\end{itemize}

Let $S$ be the line-free set obtained from $S^*$ by deleting these points; then $|S|\ge |S^*|-\ell(\ell-1)(s+1)$. Figure \ref{fig:finalset} shows the points of $S$ in grey, and the points that were excluded already from $S^*$ in red. The light blue segments are completely included in $S^*$, but due to the lines of the mentioned form, at most $(\ell-1)(s+1)$ of their points (asymptotically $\frac14p$ as $p\to \infty$) per layer are actually excluded from $S$. Note that as $p\to \infty$, we have $r,s\sim \sqrt{p}$ and $\ell\sim \frac14\sqrt{p}$. To make the construction easier to visualize, the ratio $\frac{\ell r}{p}$ appears to be larger in the figure than its actual value $\approx \frac14$.

\begin{figure}
\centering
\usetikzlibrary{patterns}
\definecolor{figLightBlue}{HTML}{56B4E9}
\begin{tikzpicture}[line cap=round,line join=round,>=triangle 45,x=1.0cm,y=1.0cm,every node/.style={font=\tiny}]
\clip(-0.9,-1.2) rectangle (16.5,3.1);
\fill[line width=0.pt,color=yqyqyq,fill=yqyqyq,fill opacity=0.4000000059604645] (0.,0.) -- (2.,0.) -- (2.,2.) -- (0.,2.) -- cycle;
\fill[line width=0.pt,color=yqyqyq,fill=yqyqyq,fill opacity=0.4000000059604645] (3.5,0.) -- (5.5,0.) -- (5.5,2.) -- (3.5,2.) -- cycle;
\fill[line width=0.pt,color=yqyqyq,fill=yqyqyq,fill opacity=0.4000000059604645] (7.,0.) -- (9.,0.) -- (9.,2.) -- (7.,2.) -- cycle;
\fill[line width=0.pt,color=yqyqyq,fill=yqyqyq,fill opacity=0.4000000059604645] (10.5,0.) -- (12.5,0.) -- (12.5,2.) -- (10.5,2.) -- cycle;
\fill[line width=0.pt,color=ffqqqq,fill=ffqqqq,fill opacity=0.4000000059604645] (14.,0.) -- (16.,0.) -- (16.,2.) -- (14.,2.) -- cycle;
\path[pattern=north east lines,pattern color=ffqqqq] (14,0) rectangle (16,2);
\draw [line width=2pt,color=ffqqqq,line cap=butt,dash pattern=on 3pt off 2pt] (0.,2.)-- (2.,2.);
\draw [line width=2pt,color=ffqqqq,line cap=butt,dash pattern=on 3pt off 2pt] (2.,2.)-- (2.,0.);
\draw [line width=2.pt,color=yqyqyq] (2.,0.)-- (0.,0.);
\draw [line width=2.pt,color=yqyqyq] (0.,0.)-- (0.,2.);
\draw (0,-0.3) node[anchor=center] {$0$};
\draw (2,-0.3) node[anchor=center] {$p-1$};
\draw (-0.05,0) node[anchor=east] {$0$};
\draw (-0.05,2) node[anchor=east] {$p-1$};
\draw (2.7,0) node[anchor=center,font=\Large] {$\cdots$};
\draw (1,-0.9) node[anchor=center,font=\small] {$S_0$};
\draw (4.5,-0.9) node[anchor=center,font=\small] {$S_{p-\ell-2}$};
\draw (8,-0.9) node[anchor=center,font=\small] {$S_{p-\ell-1}$};
\draw (11.5,-0.9) node[anchor=center,font=\small] {$S_{p-2}$};
\draw (15,-0.9) node[anchor=center,font=\small] {$S_{p-1}$};
\draw [line width=2pt,color=ffqqqq,line cap=butt,dash pattern=on 3pt off 2pt] (3.5,2.)-- (5.5,2.);
\draw [line width=2pt,color=ffqqqq,line cap=butt,dash pattern=on 3pt off 2pt] (5.5,2.)-- (5.5,0.);
\draw [line width=2.pt,color=yqyqyq] (5.5,0.)-- (3.5,0.);
\draw [line width=2.pt,color=yqyqyq] (3.5,0.)-- (3.5,2.);
\draw (3.5,-0.3) node[anchor=center] {$0$};
\draw (5.5,-0.3) node[anchor=center] {$p-1$};
\draw (3.45,0) node[anchor=east] {$0$};
\draw (3.45,2) node[anchor=east] {$p-1$};
\draw [line width=2.5pt,color=figLightBlue,line cap=round,dash pattern=on 0pt off 4pt] (9.,2.)-- (9.,0.);
\draw [line width=2.pt,color=yqyqyq] (7.,0.)-- (7.,2.);
\draw [line width=2pt,color=ffqqqq,line cap=butt,dash pattern=on 3pt off 2pt] (7.,0.)-- (7.4,0.);
\draw [line width=2.pt,color=yqyqyq] (7.4,0.)-- (9.,0.);
\draw [line width=2pt,color=ffqqqq,line cap=butt,dash pattern=on 3pt off 2pt] (7.,1.8)-- (7.4,1.8);
\draw [line width=2pt,color=ffqqqq,line cap=butt,dash pattern=on 3pt off 2pt] (7.,0.8)-- (7.4,0.8);
\draw [line width=2pt,color=ffqqqq,line cap=butt,dash pattern=on 3pt off 2pt] (7.,0.4)-- (7.4,0.4);
\draw [line width=2pt,color=ffqqqq,line cap=butt,dash pattern=on 3pt off 2pt] (7.,1.4)-- (7.4,1.4);
\draw [line width=2pt,color=ffqqqq,line cap=butt,dash pattern=on 3pt off 2pt] (7.,0.)-- (9.,2.);
\draw [line width=2pt,color=ffqqqq,line cap=butt,dash pattern=on 3pt off 2pt] (7.,2.)-- (8.2,2.);
\draw [line width=2.pt,color=yqyqyq] (8.2,2.)-- (9.,2.);
\draw (7,-0.3) node[anchor=center] {$0$};
\draw (7.4,-0.4) node[anchor=center] {$r-1$};
\draw [line width=0.1pt,color=black] (7.4,-0.31)-- (7.4,-0.1);
\draw (9,-0.3) node[anchor=center] {$p-1$};
\draw (6.95,0) node[anchor=east] {$0$};
\draw (6.95,0.4) node[anchor=east] {$r$};
\draw (6.95,0.8) node[anchor=east] {$2r$};
\draw (6.95,1.4) node[anchor=east] {$(s-1)r$};
\draw (6.95,1.8) node[anchor=east] {$sr$};
\draw (6.95,2) node[anchor=east] {$p-1$};
\draw (9.7,0) node[anchor=center,font=\Large] {$\cdots$};

\draw [line width=2.5pt,color=figLightBlue,line cap=round,dash pattern=on 0pt off 4pt] (12.5,2.)-- (12.5,0.);
\draw [line width=2.pt,color=yqyqyq] (10.5,0.)-- (10.5,2.);
\draw [line width=2pt,color=ffqqqq,line cap=butt,dash pattern=on 3pt off 2pt] (10.5,0.)-- (12.5,2.);
\draw [line width=2pt,color=ffqqqq,line cap=butt,dash pattern=on 3pt off 2pt] (10.5,2.)-- (11.7,2.);
\draw [line width=2.pt,color=yqyqyq] (11.7,2.)-- (12.5,2.);
\draw [line width=2pt,color=ffqqqq,line cap=butt,dash pattern=on 3pt off 2pt] (11.3,1.8)-- (11.7,1.8);
\draw [line width=2pt,color=ffqqqq,line cap=butt,dash pattern=on 3pt off 2pt] (11.3,0.)-- (11.7,0.);
\draw [line width=2pt,color=ffqqqq,line cap=butt,dash pattern=on 3pt off 2pt] (11.3,0.8)-- (11.7,0.8);
\draw [line width=2pt,color=ffqqqq,line cap=butt,dash pattern=on 3pt off 2pt] (11.3,0.4)-- (11.7,0.4);
\draw [line width=2pt,color=ffqqqq,line cap=butt,dash pattern=on 3pt off 2pt] (11.3,1.4)-- (11.7,1.4);
\draw [line width=2.pt,color=yqyqyq] (10.5,0.)-- (11.3,0.);
\draw [line width=2.pt,color=yqyqyq] (11.7,0.)-- (12.5,0.);
\draw (10.5,-0.3) node[anchor=center] {$0$};
\draw (12.5,-0.3) node[anchor=center] {$p-1$};
\draw (10.45,0) node[anchor=east] {$0$};
\draw (10.45,0.4) node[anchor=east] {$r$};
\draw (10.45,0.8) node[anchor=east] {$2r$};
\draw (10.45,1.4) node[anchor=east] {$(s-1)r$};
\draw (10.45,1.8) node[anchor=east] {$sr$};
\draw (10.45,2) node[anchor=east] {$p-1$};
\draw (11.3,-0.4) node[anchor=center] {$\ell r-r$};
\draw (11.7,-0.2) node[anchor=center] {$\ell r-1$};
\draw [line width=0.1pt,color=black] (11.3,-0.31)-- (11.3,-0.1);
\draw [line width=2pt,color=ffqqqq,line cap=butt,dash pattern=on 3pt off 2pt] (16.,2.)-- (16.,0.);
\draw [line width=2pt,color=ffqqqq,line cap=butt,dash pattern=on 3pt off 2pt] (14.,0.)-- (14.,2.);
\draw [line width=2pt,color=ffqqqq,line cap=butt,dash pattern=on 3pt off 2pt] (14.,2.)-- (15.2,2.);
\draw [line width=2pt,color=ffqqqq,line cap=butt,dash pattern=on 3pt off 2pt] (15.2,2.)-- (16.,2.);
\draw [line width=2pt,color=ffqqqq,line cap=butt,dash pattern=on 3pt off 2pt] (15.2,0.)-- (16.,0.);
\draw [line width=2.pt,color=yqyqyq] (14.,1.8)-- (15.2,1.8);
\draw [line width=2.pt,color=yqyqyq] (14.,1.4)-- (15.2,1.4);
\draw [line width=2.pt,color=yqyqyq] (14.,0.8)-- (15.2,0.8);
\draw [line width=2.pt,color=yqyqyq] (15.2,0.4)-- (14.,0.4);
\draw [line width=2.pt,color=yqyqyq] (14.,0.)-- (15.2,0.);
\draw (14,-0.3) node[anchor=center] {$0$};
\draw (16,-0.3) node[anchor=center] {$p-1$};
\draw (13.95,0) node[anchor=east] {$0$};
\draw (13.95,0.4) node[anchor=east] {$r$};
\draw (13.95,0.8) node[anchor=east] {$2r$};
\draw (13.95,1.4) node[anchor=east] {$(s-1)r$};
\draw (13.95,1.8) node[anchor=east] {$sr$};
\draw (13.95,2) node[anchor=east] {$p-1$};
\draw (15.2,-0.2) node[anchor=center] {$\ell r-1$};
\begin{scriptsize}
\draw [draw=black,fill=yqyqyq,line width=0.5pt] (0.,0.) circle (2pt);
\path (2.,0.) node[rectangle,draw=black,fill=ffqqqq,line width=0.4pt,minimum size=4pt,inner sep=0pt,outer sep=0pt] {};
\path (0.,2.) node[rectangle,draw=black,fill=ffqqqq,line width=0.4pt,minimum size=4pt,inner sep=0pt,outer sep=0pt] {};
\path (2.,2.) node[rectangle,draw=black,fill=ffqqqq,line width=0.4pt,minimum size=4pt,inner sep=0pt,outer sep=0pt] {};
\draw [draw=black,fill=yqyqyq,line width=0.5pt] (3.5,0.) circle (2pt);
\path (5.5,0.) node[rectangle,draw=black,fill=ffqqqq,line width=0.4pt,minimum size=4pt,inner sep=0pt,outer sep=0pt] {};
\path (3.5,2.) node[rectangle,draw=black,fill=ffqqqq,line width=0.4pt,minimum size=4pt,inner sep=0pt,outer sep=0pt] {};
\path (5.5,2.) node[rectangle,draw=black,fill=ffqqqq,line width=0.4pt,minimum size=4pt,inner sep=0pt,outer sep=0pt] {};
\path (7.,0.) node[rectangle,draw=black,fill=ffqqqq,line width=0.4pt,minimum size=4pt,inner sep=0pt,outer sep=0pt] {};
\path (9.,0.) node[rectangle,draw=black,fill=ffqqqq,line width=0.4pt,minimum size=4pt,inner sep=0pt,outer sep=0pt] {};
\path (9.,2.) node[rectangle,draw=black,fill=ffqqqq,line width=0.4pt,minimum size=4pt,inner sep=0pt,outer sep=0pt] {};
\path (7.,2.) node[rectangle,draw=black,fill=ffqqqq,line width=0.4pt,minimum size=4pt,inner sep=0pt,outer sep=0pt] {};
\path (7.,2.) node[rectangle,draw=black,fill=ffqqqq,line width=0.4pt,minimum size=4pt,inner sep=0pt,outer sep=0pt] {};
\path (9.,2.) node[rectangle,draw=black,fill=ffqqqq,line width=0.4pt,minimum size=4pt,inner sep=0pt,outer sep=0pt] {};
\path (9.,2.) node[rectangle,draw=black,fill=ffqqqq,line width=0.4pt,minimum size=4pt,inner sep=0pt,outer sep=0pt] {};
\path (9.,0.) node[rectangle,draw=black,fill=ffqqqq,line width=0.4pt,minimum size=4pt,inner sep=0pt,outer sep=0pt] {};
\path (9.,0.) node[rectangle,draw=black,fill=ffqqqq,line width=0.4pt,minimum size=4pt,inner sep=0pt,outer sep=0pt] {};
\path (7.,0.) node[rectangle,draw=black,fill=ffqqqq,line width=0.4pt,minimum size=4pt,inner sep=0pt,outer sep=0pt] {};
\path (7.,2.) node[rectangle,draw=black,fill=ffqqqq,line width=0.4pt,minimum size=4pt,inner sep=0pt,outer sep=0pt] {};
\path (7.4,0.) node[rectangle,draw=black,fill=ffqqqq,line width=0.4pt,minimum size=4pt,inner sep=0pt,outer sep=0pt] {};
\path (7.,0.4) node[rectangle,draw=black,fill=ffqqqq,line width=0.4pt,minimum size=4pt,inner sep=0pt,outer sep=0pt] {};
\path (7.4,0.4) node[rectangle,draw=black,fill=ffqqqq,line width=0.4pt,minimum size=4pt,inner sep=0pt,outer sep=0pt] {};
\path (7.,0.8) node[rectangle,draw=black,fill=ffqqqq,line width=0.4pt,minimum size=4pt,inner sep=0pt,outer sep=0pt] {};
\path (7.4,0.8) node[rectangle,draw=black,fill=ffqqqq,line width=0.4pt,minimum size=4pt,inner sep=0pt,outer sep=0pt] {};
\path (7.4,1.8) node[rectangle,draw=black,fill=ffqqqq,line width=0.4pt,minimum size=4pt,inner sep=0pt,outer sep=0pt] {};
\path (7.,1.8) node[rectangle,draw=black,fill=ffqqqq,line width=0.4pt,minimum size=4pt,inner sep=0pt,outer sep=0pt] {};
\path (7.,1.4) node[rectangle,draw=black,fill=ffqqqq,line width=0.4pt,minimum size=4pt,inner sep=0pt,outer sep=0pt] {};
\path (7.4,1.4) node[rectangle,draw=black,fill=ffqqqq,line width=0.4pt,minimum size=4pt,inner sep=0pt,outer sep=0pt] {};
\path (8.2,2.) node[rectangle,draw=black,fill=ffqqqq,line width=0.4pt,minimum size=4pt,inner sep=0pt,outer sep=0pt] {};
\path (9.,0.4) node[rectangle,draw=black,fill=ffqqqq,line width=0.4pt,minimum size=4pt,inner sep=0pt,outer sep=0pt] {};
\path (9.,0.8) node[rectangle,draw=black,fill=ffqqqq,line width=0.4pt,minimum size=4pt,inner sep=0pt,outer sep=0pt] {};
\path (9.,1.8) node[rectangle,draw=black,fill=ffqqqq,line width=0.4pt,minimum size=4pt,inner sep=0pt,outer sep=0pt] {};
\path (9.,1.4) node[rectangle,draw=black,fill=ffqqqq,line width=0.4pt,minimum size=4pt,inner sep=0pt,outer sep=0pt] {};

\path (10.5,0.) node[rectangle,draw=black,fill=ffqqqq,line width=0.4pt,minimum size=4pt,inner sep=0pt,outer sep=0pt] {};
\path (12.5,0.) node[rectangle,draw=black,fill=ffqqqq,line width=0.4pt,minimum size=4pt,inner sep=0pt,outer sep=0pt] {};
\path (12.5,2.) node[rectangle,draw=black,fill=ffqqqq,line width=0.4pt,minimum size=4pt,inner sep=0pt,outer sep=0pt] {};
\path (10.5,2.) node[rectangle,draw=black,fill=ffqqqq,line width=0.4pt,minimum size=4pt,inner sep=0pt,outer sep=0pt] {};
\path (12.5,2.) node[rectangle,draw=black,fill=ffqqqq,line width=0.4pt,minimum size=4pt,inner sep=0pt,outer sep=0pt] {};
\path (12.5,0.) node[rectangle,draw=black,fill=ffqqqq,line width=0.4pt,minimum size=4pt,inner sep=0pt,outer sep=0pt] {};
\path (10.5,2.) node[rectangle,draw=black,fill=ffqqqq,line width=0.4pt,minimum size=4pt,inner sep=0pt,outer sep=0pt] {};
\path (10.5,0.) node[rectangle,draw=black,fill=ffqqqq,line width=0.4pt,minimum size=4pt,inner sep=0pt,outer sep=0pt] {};
\path (12.5,0.) node[rectangle,draw=black,fill=ffqqqq,line width=0.4pt,minimum size=4pt,inner sep=0pt,outer sep=0pt] {};
\path (10.5,0.) node[rectangle,draw=black,fill=ffqqqq,line width=0.4pt,minimum size=4pt,inner sep=0pt,outer sep=0pt] {};
\path (12.5,2.) node[rectangle,draw=black,fill=ffqqqq,line width=0.4pt,minimum size=4pt,inner sep=0pt,outer sep=0pt] {};
\path (10.5,2.) node[rectangle,draw=black,fill=ffqqqq,line width=0.4pt,minimum size=4pt,inner sep=0pt,outer sep=0pt] {};
\path (11.7,2.) node[rectangle,draw=black,fill=ffqqqq,line width=0.4pt,minimum size=4pt,inner sep=0pt,outer sep=0pt] {};
\path (11.7,2.) node[rectangle,draw=black,fill=ffqqqq,line width=0.4pt,minimum size=4pt,inner sep=0pt,outer sep=0pt] {};
\path (12.5,2.) node[rectangle,draw=black,fill=ffqqqq,line width=0.4pt,minimum size=4pt,inner sep=0pt,outer sep=0pt] {};
\path (10.5,2.) node[rectangle,draw=black,fill=ffqqqq,line width=0.4pt,minimum size=4pt,inner sep=0pt,outer sep=0pt] {};
\path (12.5,2.) node[rectangle,draw=black,fill=ffqqqq,line width=0.4pt,minimum size=4pt,inner sep=0pt,outer sep=0pt] {};
\path (12.5,0.) node[rectangle,draw=black,fill=ffqqqq,line width=0.4pt,minimum size=4pt,inner sep=0pt,outer sep=0pt] {};
\path (10.5,0.) node[rectangle,draw=black,fill=ffqqqq,line width=0.4pt,minimum size=4pt,inner sep=0pt,outer sep=0pt] {};
\path (12.5,0.4) node[rectangle,draw=black,fill=ffqqqq,line width=0.4pt,minimum size=4pt,inner sep=0pt,outer sep=0pt] {};
\path (12.5,0.8) node[rectangle,draw=black,fill=ffqqqq,line width=0.4pt,minimum size=4pt,inner sep=0pt,outer sep=0pt] {};
\path (12.5,1.8) node[rectangle,draw=black,fill=ffqqqq,line width=0.4pt,minimum size=4pt,inner sep=0pt,outer sep=0pt] {};
\path (12.5,1.4) node[rectangle,draw=black,fill=ffqqqq,line width=0.4pt,minimum size=4pt,inner sep=0pt,outer sep=0pt] {};
\path (11.3,1.8) node[rectangle,draw=black,fill=ffqqqq,line width=0.4pt,minimum size=4pt,inner sep=0pt,outer sep=0pt] {};
\path (11.7,1.8) node[rectangle,draw=black,fill=ffqqqq,line width=0.4pt,minimum size=4pt,inner sep=0pt,outer sep=0pt] {};
\path (11.3,0.) node[rectangle,draw=black,fill=ffqqqq,line width=0.4pt,minimum size=4pt,inner sep=0pt,outer sep=0pt] {};
\path (11.7,0.) node[rectangle,draw=black,fill=ffqqqq,line width=0.4pt,minimum size=4pt,inner sep=0pt,outer sep=0pt] {};
\path (11.3,0.8) node[rectangle,draw=black,fill=ffqqqq,line width=0.4pt,minimum size=4pt,inner sep=0pt,outer sep=0pt] {};
\path (11.7,0.8) node[rectangle,draw=black,fill=ffqqqq,line width=0.4pt,minimum size=4pt,inner sep=0pt,outer sep=0pt] {};
\path (11.3,0.4) node[rectangle,draw=black,fill=ffqqqq,line width=0.4pt,minimum size=4pt,inner sep=0pt,outer sep=0pt] {};
\path (11.7,0.4) node[rectangle,draw=black,fill=ffqqqq,line width=0.4pt,minimum size=4pt,inner sep=0pt,outer sep=0pt] {};
\path (11.3,1.4) node[rectangle,draw=black,fill=ffqqqq,line width=0.4pt,minimum size=4pt,inner sep=0pt,outer sep=0pt] {};
\path (11.7,1.4) node[rectangle,draw=black,fill=ffqqqq,line width=0.4pt,minimum size=4pt,inner sep=0pt,outer sep=0pt] {};
\draw [draw=black,fill=yqyqyq,line width=0.5pt] (11.3,0.) circle (2pt);
\path (11.3,0.) node[rectangle,draw=black,fill=ffqqqq,line width=0.4pt,minimum size=4pt,inner sep=0pt,outer sep=0pt] {};

\draw [draw=black,fill=yqyqyq,line width=0.5pt] (14.,0.) circle (2pt);
\path (16.,0.) node[rectangle,draw=black,fill=ffqqqq,line width=0.4pt,minimum size=4pt,inner sep=0pt,outer sep=0pt] {};
\path (16.,2.) node[rectangle,draw=black,fill=ffqqqq,line width=0.4pt,minimum size=4pt,inner sep=0pt,outer sep=0pt] {};
\path (14.,2.) node[rectangle,draw=black,fill=ffqqqq,line width=0.4pt,minimum size=4pt,inner sep=0pt,outer sep=0pt] {};
\path (16.,2.) node[rectangle,draw=black,fill=ffqqqq,line width=0.4pt,minimum size=4pt,inner sep=0pt,outer sep=0pt] {};
\path (16.,0.) node[rectangle,draw=black,fill=ffqqqq,line width=0.4pt,minimum size=4pt,inner sep=0pt,outer sep=0pt] {};
\draw [draw=black,fill=yqyqyq,line width=0.5pt] (14.,0.) circle (2pt);
\path (14.,2.) node[rectangle,draw=black,fill=ffqqqq,line width=0.4pt,minimum size=4pt,inner sep=0pt,outer sep=0pt] {};
\draw [draw=black,fill=yqyqyq,line width=0.5pt] (14.,0.) circle (2pt);
\path (16.,2.) node[rectangle,draw=black,fill=ffqqqq,line width=0.4pt,minimum size=4pt,inner sep=0pt,outer sep=0pt] {};
\path (14.,2.) node[rectangle,draw=black,fill=ffqqqq,line width=0.4pt,minimum size=4pt,inner sep=0pt,outer sep=0pt] {};
\path (15.2,2.) node[rectangle,draw=black,fill=ffqqqq,line width=0.4pt,minimum size=4pt,inner sep=0pt,outer sep=0pt] {};
\path (15.2,2.) node[rectangle,draw=black,fill=ffqqqq,line width=0.4pt,minimum size=4pt,inner sep=0pt,outer sep=0pt] {};
\path (16.,2.) node[rectangle,draw=black,fill=ffqqqq,line width=0.4pt,minimum size=4pt,inner sep=0pt,outer sep=0pt] {};
\draw [draw=black,fill=yqyqyq,line width=0.5pt] (15.2,1.8) circle (2pt);
\draw [draw=black,fill=yqyqyq,line width=0.5pt] (15.2,0.8) circle (2pt);
\draw [draw=black,fill=yqyqyq,line width=0.5pt] (15.2,0.4) circle (2pt);
\draw [draw=black,fill=yqyqyq,line width=0.5pt] (15.2,1.4) circle (2pt);
\draw [draw=black,fill=yqyqyq,line width=0.5pt] (15.2,0.) circle (2pt);
\path (16.,0.) node[rectangle,draw=black,fill=ffqqqq,line width=0.4pt,minimum size=4pt,inner sep=0pt,outer sep=0pt] {};
\path (16.,0.) node[rectangle,draw=black,fill=ffqqqq,line width=0.4pt,minimum size=4pt,inner sep=0pt,outer sep=0pt] {};
\path (14.,2.) node[rectangle,draw=black,fill=ffqqqq,line width=0.4pt,minimum size=4pt,inner sep=0pt,outer sep=0pt] {};
\path (16.,2.) node[rectangle,draw=black,fill=ffqqqq,line width=0.4pt,minimum size=4pt,inner sep=0pt,outer sep=0pt] {};
\path (16.,0.) node[rectangle,draw=black,fill=ffqqqq,line width=0.4pt,minimum size=4pt,inner sep=0pt,outer sep=0pt] {};
\draw [draw=black,fill=yqyqyq,line width=0.5pt] (14.,0.) circle (2pt);
\draw [draw=black,fill=yqyqyq,line width=0.5pt] (14.,0.4) circle (2pt);
\draw [draw=black,fill=yqyqyq,line width=0.5pt] (14.,0.8) circle (2pt);
\draw [draw=black,fill=yqyqyq,line width=0.5pt] (14.,1.4) circle (2pt);
\draw [draw=black,fill=yqyqyq,line width=0.5pt] (14.,1.8) circle (2pt);
\end{scriptsize}
\end{tikzpicture}

\captionsetup{skip=1pt}
\caption{Layers of the final line-free set $S$. Grey points are included in $S$, whereas red points are excluded already from $S^*$ (and hence also from $S$). Grey points are indicated with circles, solid lines and plain shading, whereas red points are indicated with squares, dashed lines and hatched shading. The light blue points, indicated with dotted segments, are all included in $S^*$, but at most $(\ell-1)(s+1)$ of them are removed from each layer when obtaining $S$ from $S^*$.}
\label{fig:finalset}
\end{figure}

We now calculate the size of $S^*$, counting its points layer by layer.

\begin{itemize}

\item For $i\in [0,p-\ell-2]$, $S^*_i$ is a standard square, containing $(p-1)^2$ points.

\item For $i\in [p-\ell-1, p-2]$, the set $A_i$ contains $p$ points on the main diagonal, $(s+1)(r+1)$ points in the grid, and $\ell r$ points with second coordinate $p-1$. A point $(j,j)$ on the main diagonal appears in the grid as well if we have $r\mid j$ and $t\le j\le t+r-1$. (Note that we cannot have $j=p-1$, since $t+r-1\le \ell r-1<p-1$.) These two conditions are satisfied for exactly one value $j$, so $|A_i|=p+(s+1)(r+1)+\ell r-1$ and $|S^*_i|=p^2-p-(s+1)(r+1)-\ell r+1$.

\item In layer $p-1$, we have $|S^*_{p-1}|=(s+1)\cdot \ell r$.
\end{itemize}

Altogether, 
$$|S^*|=(p-\ell-1)(p-1)^2+\ell(p^2-p-(s+1)(r+1)-\ell r+1)+(s+1)\ell r,$$

therefore we have
$$|S|\ge |S^*|-\ell(\ell-1)(s+1)\ge (p-\ell-1)(p-1)^2+\ell(p^2-p-(s+1)(r+\ell)-\ell r+1)+(s+1)\ell r.$$

Using $r=\left \lfloor \sqrt{p}\right \rfloor$, $s=\left \lfloor \frac{p-2}{r}\right \rfloor$ and $\ell=\left\lfloor \frac14\sqrt{p}\right\rfloor$, we now put a lower bound on this size. For $p\ge 5$, the parameters can be bounded by
$$\sqrt{p}-1\le r\le \sqrt{p},$$
$$\sqrt{p}-2\le \frac{p-2}{\sqrt{p}}-1\le \left\lfloor \frac{p-2}{r}\right\rfloor = s\le \frac{p-2}{r}\le \frac{p}{\sqrt{p}-1}\le \sqrt{p}+2,$$
$$\frac14\sqrt{p}-1\le \ell\le \frac14\sqrt{p}.$$

So we have
\begin{equation*}
\begin{split}
|S| & \ge (p-\ell-1)(p^2-2p)+\ell\left(p^2-p-(\sqrt{p}+3)\cdot \frac54\sqrt{p}-\frac14p\right)+(\sqrt{p}-1)^2\left(\frac14\sqrt{p}-1\right) \\
& \ge p^3-2p^2-\ell p^2+2\ell p-p^2+2p+\ell\left(p^2-\frac52p-\frac{15}{4}\sqrt{p}\right)+\frac14p^{3/2}-\frac32p\\
& = p^3-3p^2-\frac12\ell p-\frac{15}{4}\ell \sqrt{p}+\frac14p^{3/2}+\frac12p\\
& \ge p^3-3p^2-\frac18p^{3/2}-\frac{15}{16}p+\frac14p^{3/2}+\frac12p=p^3-3p^2+\frac18p^{3/2}-\frac{7}{16}p=(p-1)^3+\frac18p^{3/2}-O(p),
\end{split}
\end{equation*}

that is, our line-free set $S$ indeed has at least the claimed size.
\end{proof}

\begin{remark}
The number $\ell$ of special layers in the construction has been optimized to give the best coefficient for the $p^{3/2}$ term in the resulting lower bound. If we take $\ell\sim c\sqrt{p}$ for some constant $0< c\le 1$, then the total number of columns in the grids of special layers will be $\ell r\sim cp$, and our lower bound becomes $|S|\ge p^3-3p^2+(c-2c^2)p^{3/2}-o(p^{3/2})$, where $c-2c^2$ is maximized at $c=\frac14$ with a value of $\frac18$.
\end{remark}

\section*{Acknowledgements}
The author would like to thank the anonymous referees for their helpful suggestions regarding the presentation of the manuscript. The author is supported by the EKÖP-25 University Research Scholarship Program of the Ministry for Culture and Innovation, and grant number ADVANCED 153080, from the source of the National Research, Development and Innovation Fund.

\end{document}